\documentclass[12pt,a4paper]{amsart}
\usepackage{amssymb,amsmath,amsthm}
\usepackage{graphicx}
\usepackage{hyperref}
 \usepackage[color=green!40]{todonotes}

\newcommand\cT{{\mathcal T}}

\newcommand{\ex}{\mathop{}\!\mathrm{ex}}

\makeatletter
\newtheorem*{rep@theorem}{\rep@title}
\newcommand{\newreptheorem}[2]{%
\newenvironment{rep#1}[1]{%
 \def\rep@title{#2 \ref{##1}}%
 \begin{rep@theorem}}%
 {\end{rep@theorem}}}
\makeatother

\theoremstyle{plain}
\newtheorem{theorem}{Theorem}[section]
\newreptheorem{theorem}{Theorem}

\newtheorem{proposition}[theorem]{Proposition}

\newtheorem{problem}[theorem]{Problem}
\theoremstyle{definition}

\newtheorem{defn}[theorem]{Definition}

\newtheorem{claim}[theorem]{Claim}

\newtheorem{rem}[theorem]{Remark}

\newcommand\cref[1]{Corollary~\ref{cor:#1}}

\title{Connected Turán number of trees}

\author{Yair Caro}
\address{Department of Mathematics, University of Haifa-Oranim, Israel}
\email{yacaro@kvgeva.org.il}

\author{Bal\'azs Patk\'os}
\address{Alfr\'ed R\'enyi Institute of Mathematics}
\email{patkos@renyi.hu}
\thanks{Patk\'os's research is partially supported by NKFIH grants SNN 129364 and FK 132060.}

\author{Zsolt Tuza}
\address{Alfr\'ed R\'enyi Institute of Mathematics and University of Pannonia}
\email{tuza@dcs.uni-pannon.hu}
\thanks{Tuza's research is partially supported by NKFIH grant SNN 129364.}
\date{}

\begin{document}

\begin{abstract}
    As a variant of the much studied Tur\'an number, $\ex(n,F)$, the largest number of edges that an $n$-vertex $F$-free graph may contain, we introduce the connected Tur\'an number $\ex_c(n,F)$, the largest number of edges that an $n$-vertex connected $F$-free graph may contain. We focus on the case where the forbidden graph is a tree. The celebrated conjecture of Erd\H os and S\'os states that for any tree $T$, we have $\ex(n,T)\le(|T|-2)\frac{n}{2}$. We address the problem how much smaller $\ex_c(n,T)$ can be, what is the smallest possible ratio of $\ex_c(n,T)$ and $(|T|-2)\frac{n}{2}$ as $|T|$ grows. We also determine the exact value of $\ex_c(n,T)$ for small trees, in particular for all trees with at most six vertices. We introduce general constructions of connected $T$-free graphs based on graph parameters as longest path, matching number, branching number, etc.
\end{abstract}

\maketitle

\section{Introduction}

One of the most studied problems in extremal graph theory is to determine the Tur\'an number $\ex(n,F)$, the largest number of edges that an $n$-vertex graph can have without containing a subgraph isomorphic to $F$. In this paper, we study a variant of this parameter: the connected Tur\'an number $\ex_c(n,F)$ is the largest number of edges that a \textit{connected} $n$-vertex graph can have without containing $F$ as a subgraph. Observe that if $F$ is 2-edge-connected, then any maximal $F$-free graph $G$ is connected, as if $G$ had at least two components, then adding an edge between them would not create any copy of $F$. Also, if the chromatic number of $F$ is at least 3, then by the famous theorem by Erd\H os, Stone, and Simonovits \cite{ES,ESt}, we know that $\ex(n,F)$ is attained asymptotically (and for some graphs precisely) at the Tur\'an graph that is connected. These two observations imply the following proposition.

\begin{proposition}\
\begin{enumerate}
    \item 
    If all components of\/ $F$ are 2-edge-connected, then\/ $\ex(n,F)=\ex_c(n,F)$.
    \item
    If\/ $\chi(F)\ge 3$, then\/ $\ex_c(n,F)=(1+o(1))\ex(n,F)$.
\end{enumerate}
\end{proposition}

The asymptotics of $\ex(n,F)$ is unknown for most biparite $F$ (for a general overview of the so-called degenrate Tur\'an problems, see the survey by F\"uredi and Simonovits \cite{FS}). And we do not know the relationship of $\ex(n,F)$ and $\ex_c(n,F)$ for most bipartite $F$ that are not 2-edge-connected. There is a relatively large literature on the Tur\'an number of forests (see e.g.\ \cite{BK,LLST, LLP,YZ,ZW}), and in many cases the extremal graphs turned out to be connected, so for those forests $F$, we have $\ex(n,F)=\ex_c(n,F)$. A wide and important class of connected non-2-edge-connected graphs is the set of trees. A famous conjecture of Erd\H os and S\'os (that appeared in print first in \cite{E}) states that any $n$-vertex graph with more than $\frac{(k-2)n}{2}$ edges contains any tree $T$ on $k$ vertices. A proof was announced in the early 1990's by Ajtai, Koml\'os, Simonovits, and Szemer\'edi, but only arguments of special cases have appeared. A recent survey of these and other degree conditions that imply embeddings of trees is \cite{S}. The universal construction that shows the tigthness of the Erd\H os--S\'os conjecture is the union of vertex-disjoint cliques of size $k-1$. This is not a connected graph and we are only aware of one result concerning $\ex_c(n,T)$ (but there exist results on Tur\'an problems in connected host graphs, see e.g.\ \cite{BJ}). We denote by $P_k$ the path on $k$ vertices. The value of $\ex_c(n,P_k)$ was determined by Kopylov, and independently by Balister, Gy\H ori, Lehel, and Schelp with the latter group also showing the uniqueness of extremal constructions.

\begin{theorem}[Kopylov \cite{K}, Balister, Gy\H ori, Lehel, Schelp \cite{BGLS}]\label{paths}
If\/ $G$ is an\/ $n$-vertex connected graph that does not contain any paths on\/ $k+1$ vertices, then $$e(G)\le \max\left\{\binom{k-1}{2}+n-k+1,\binom{\lceil \frac{k+1}{2}\rceil}{2}+\left\lfloor \frac{k-1}{2}\right\rfloor\left(n-\left\lceil \frac{k+1}{2}\right\rceil\right)\right\}$$ holds.
\end{theorem}

We shall now present the various results obtained concerning $\ex_c(n ,T)$. Lower bound  constructions are given in Section  2   and exact determination of $\ex_c(n,T)$ including all trees up to 6 vertices is included in Section 3.

Our first result gathers several constructions, all based on some graph parameters, that provide lower bounds on $\ex_c(n,T)$.
For those parameters we use the following notation.

\begin{defn}\
\begin{itemize}
    \item 
$\ell(G)$ denotes the number of vertices in a longest path in $G$.
\item
$p(G)$ denotes the maximum number of vertices in a path $P$ of $G$ such that for all $x\in V(P)$ we have $d_G(x)\le 2$.
\item
$\Delta(G)$ and $\delta(G)$ denote the maximum and the minimum degree in $G$.
\item
$\nu(G)$ denotes the number of edges in a largest matching of $G$.
\item
$\delta_2(T)$ denotes the smallest degree in $T$ that is larger than 1.
\item
For a vertex $v\in V(T)$ let $m_T(v)$ be the size of largest component of $T - v$ and let $m(T)=\min\{m_T(v):v\in V(T)\}$. 
\item
For a vertex $v\in V(T)$ let $m_{T,2}(v)$ be the sum of the sizes of two largest components of $T - v$ and let $m_2(T)=\min\{m_{T,2}(v):v\in V(T)\}$.
\item
For an edge $e=xy \in E(G)$ we write $w(e)=\min\{d_G(x),d_G(y)\}$ and define $w(G)=\max\{w(e):e\in E(G)\}$.
\end{itemize}
\end{defn}

\begin{proposition}\label{parameter} \
Suppose\/ $T$ is a tree on\/ $k\ge 4$ vertices.
\begin{enumerate}
    \item\label{longestpath}
     $\ex_c(n,T)\ge \binom{\lceil \frac{\ell(T)}{2}\rceil}{2}+ \lfloor \frac{\ell(T)-2}{2}\rfloor(n-\lceil \frac{\ell(T)}{2}\rceil)$.
    \item\label{inducedpath}
     $\ex_c(n,T)\ge (\binom{k-2p(T)-3}{2}+p(T)+2)\lfloor\frac{n}{k-p(T)-2}\rfloor$. Furthermore, if\/ $T$ contains at least two vertices of degree at least three, then\/ $\ex_c(n,T)\ge \frac{\binom{k-p(T)-1}{2}+p(T)+2}{k}n-O(k)$.
    \item\label{maxdeg}
     $\ex_c(n,T)\ge \lfloor \frac{n(\Delta(T)-1)}{2}\rfloor$.
    \item\label{matching}
     $\ex_c(n,T)\ge (\nu(T)-1)(n-\nu(T)+1)+\binom{\nu(T)-1}{2}$.
    \item\label{secondsmallest}
     If\/ $T$ is not a star and\/ $\delta_2(T)>2$, then\/ $\ex_c(n,T)\ge \lfloor \frac{n-1}{k-1}\rfloor(\binom{k-2}{2}+\delta_2(T)-1)$.
    \item\label{bipartition}
    If the bipartition of\/ $T$ consists of classes of sizes\/ $a$ and\/ $b$ with\/ $a\le b$, then\/ $\ex_c(n,T)\ge (a-1)(n-a+1)$.
    \item\label{largestbranch}
     If\/ $T$ is not a path, then\/ $\ex_c(n,T)\ge n-1+ \lfloor \frac{n-1}{m(T)-1}\rfloor \binom{m(T)-1}{2}$.
    \item\label{twobranches}
     $\ex_c(n,T)\ge \lfloor \frac{n}{k-m_2(T)}\rfloor(1+\binom{k-m_2(T)}{2})$.
     \item\label{weight}
     $\ex_c(n,T)\ge (w(T)-1)(n-w(T)+1)$.
\end{enumerate}
\end{proposition}

According to the Erd\H os--S\'os conjecture, $\ex(n,T)=\frac{k-2}{2}n+O_k(1)$. We would like to know how much smaller $\ex_c(n,T)$ can be than $\ex(n,T)$. For any tree $T$ we introduce
$$\gamma_T:=\limsup_n\frac{2}{|T|-2}\,\frac{\ex_c(n,T)}{n}$$
where $|T|$ denotes the number of vertices in $T$.
It is well-known that any graph with average degree at least $2d$ contains a subgraph with minimum degree at least $d$. Also, any tree on $k$ vertices can be embedded to any graph with minimum degree at least $k$. This shows that $\gamma_T\le 2$ for any tree $T$ on $k$ vertices. The Erd\H os--S\'os conjecture would imply $\gamma_T\le 1$.

Let $\cT_k$ denote the set of trees on at least $k$ vertices. We write $\gamma_k:=\inf\{\gamma_T:T\in \cT_k\}$ and $\gamma:=\lim_{k\rightarrow \infty}\gamma_k$ (the limit exists as $\gamma_k$ is monotone increasing).

\begin{theorem}\label{gamma}
The following upper and lower bounds hold:\/ $\frac{1}{3}\le \gamma\le \frac{2}{3}$.
\end{theorem}

Finally, we determine $\ex_c(n,T)$ for all trees on $k$ vertices with $4\le k \le 6$ (note that there do not exist $P_3$-free connected graphs), and some trees on 7 vertices. We need some notation first.

$D_{a,b}$ denotes the \textit{double star} on $a+b+2$ vertices such that the two non-leaf vertices have degree $a+1$ and $b+1$. The \textit{star with $k$ leaves} is denoted by $S_k$. $S_{a_1,a_2,\dots, a_j}$ with $j\ge 3$ denotes the \textit{spider} obtained from $j$ paths with $a_1,a_2,\dots,a_j$ edges by identifying one endpoint of all paths. So $S_{a_1,a_2,\dots,a_j}$ has $1+\sum_{i=1}^ja_i$ vertices and maximum degree $j$. The only vertex of degree at least 3 is the \textit{center} of the spider, the maximal paths starting at the center are the \textit{legs} of the spider. $M_n$ denotes the matching on $n$ vertices (so if $n$ is odd, then an isolated vertex and $\lfloor \frac{n}{2}\rfloor$ isolated edges).

For graphs $H$ and $G$, their join is denoted by $H+G$, their disjoint union is denoted by $H\cup G$. For a graph $H$ and a positive integer $k$, $kH$ denotes the pairwise vertex-disjoint union of $k$ copies of $H$.

\vskip 0.3truecm

The values of $\ex_c(n,P_{k+1})$ were determined by Theorem \ref{paths}, and for $k\ge 3$, the statement $\ex_c(n,S_k)=\lfloor \frac{n(k-1)}{2}\rfloor$ follows from Proposition \ref{parameter} (\ref{maxdeg}) and that the degree-sum of an $S_k$-free graph is at most $n(k-1)$. So in the next theorem, we only list those trees that are neither paths nor stars. In particular, all trees have 5 or 6 vertices.

\begin{theorem}\label{smalltrees}
For non-star, non-path trees with 5 or 6 vertices, the following exact results are valid.
\begin{enumerate}
    \item\label{s211}
    For any\/ $T=S_{2,1,\dots,1}$ we have\/ $\ex_c(n,T)=\lfloor \frac{n(\Delta(T)-1)}{2}\rfloor$ if\/ $n\ge |T|$. In particular,\/ $\ex_c(n,S_{2,1,1})=n$ if\/ $n\ge 5$ and\/ $\ex_c(n,S_{2,1,1,1})=\lfloor \frac{3n}{2}\rfloor$ if\/ $n\ge 6$.
    \item\label{d22}
    We have\/ $\ex_c(n,D_{2,2})=2n-4$ if\/ $n\ge 6$.
    \item\label{s311}
    We have\/ $\ex_c(n,S_{3,1,1})=\lfloor  \frac{3(n-1)}{2}\rfloor$ if\/ $n\ge 7$ and\/ $\ex(6,S_{3,1,1})=9$.
    \item\label{s221}
    We have\/ $\ex_c(n,S_{2,2,1})=2n-3$ if\/ $n\ge 6$.
\end{enumerate}
\end{theorem}

\begin{table}[h!]
\begin{tabular}{ |c|c|c|c| } 
 \hline
 Number of vertices & Tree & $\ex_c(n,T)$ & Construction \\
 \hline
 4 & $P_4$ &  $n-1$ & $S_{n-1}$\\
 \hline
  & $S_3$ & $n$ & $C_n$ \\
 \hline
 5 & $P_5$ & $n$ & $K_1 + (K_2 \cup E_{n-3})$\\
 \hline
  & $S_4$ & $\lfloor \frac{3n}{2}\rfloor$ & (nearly) 3-regular\\
 \hline
  & $S_{2,1,1}$ & $n$ & $C_n$\\
 \hline
 6 & $P_6$ & $2n-3$ & $K_2+E_{n-2}$ \\
 \hline
  & $S_5$ & $2n$ & 4-regular\\
 \hline
  & $S_{2,1,1,1}$ & $\lfloor \frac{3n}{2}\rfloor$ & (nearly) 3-regular \\
 \hline
  & $S_{2,2,1}$ & $2n-3$ & $K_2+E_{n-2}$\\
 \hline
  & $S_{3,1,1}$ & $\lfloor \frac{3(n-1)}{2}\rfloor$ & $K_1+M_{n-1}$\\
 \hline
  & $D_{2,2}$ &  $2n-4$ & $K_{2,n-2}$\\
 \hline
\end{tabular}

\caption{The value of $\ex_c(n,T)$ for all trees up to 6 vertices}
\end{table}

\begin{table}[h!]
\begin{tabular}{ |c|c|c||c|c|c| } 
 \hline
 Tree & $\ex_c(n,T)$ & Construction & Tree & $\ex_c(n,T)$ & Construction \\
 \hline
 $S_6$ & $\lfloor\frac{5n}{2}\rfloor$ & (nearly) 5-regular & $P_7$ & $2n-2$ & $K_2+(E_{n-4}\cup K_2)$\\
 \hline
 $S_{4,1,1}$ & $\ge 2n-3$ & $K_2+E_{n-2}$ & $S_{3,2,1}$ & $2n-3$ & $K_2+E_{n-2}$ \\
 \hline
 $S_{3,1,1,1}$ & $\lfloor \frac{3n}{2}\rfloor $ & (nearly) 3-regular & $S_{2,1,1,1,1}$ & $2n$ & 4-regular\\ 
 \hline
 $S_{2,2,2}$ & $2n-2$ & $K_2+(E_{n-4}\cup K_2)$ & $S_{2,2,1,1}$ & $\ge 2n-3$ & $K_2+E_{n-2}$ \\
 \hline
  $D_{2,2}^*$ & $2n-3$ & $K_2+E_{n-2}$ & $D_{2,3}$ & $\ge 2n-4$ & $K_{2,n-2}$  \\
  \hline
  $SD_{2,2}$ & $\ge \frac{13n}{7}-O(1)$ & Prop. \ref{smalltrees} (2) & $D_{2,3}$ & $\ge 2n-2$ if $6|n-1$ & Prop \ref{smalltrees} (\ref{secondsmallest})\\
 \hline
\end{tabular}
\caption{Exact values and lower bounds on $\ex_c(n,T)$ for trees with 7 vertices}
\end{table}

Let $D^*_{2,2}$ be the tree obtained from $D_{2,2}$ by attaching a leaf to one leaf of $D_{2,2}$.

\begin{theorem}\label{d22*}
 We have\/ $ex_c(D^*_{2,2}) = 2n-3$ for all\/ $n\ge 7$, and\/ $ex_c(D^*_{2,2})={n \choose 2}$ for\/ $1\le n\le 6$.
 \end{theorem}

\begin{theorem}\label{s222}
We have\/ $ex_c(S_{2,2,2}) = 2n-2$ for all\/ $n\ge 7$, and\/ $ex_c(S_{2,2,2})={n \choose 2}$ for\/ $1\le n\le 6$.
\end{theorem}

\begin{theorem}\label{s321}
We have\/ $ex_c(S_{3,2,1}) = 2n-3$ for all\/ $n\ge 7$, and\/ $ex_c(S_{3,2,1})={n \choose 2}$ for\/ $1\le n\le 6$.
\end{theorem}

\begin{theorem}\label{s31dots1}
For any\/ $T=S_{3,1,\dots,1}$ with\/ $\Delta(T)\ge 4$, we have\/ $\ex_c(n,T)=\lfloor \frac{(\Delta(T)-1)n}{2}\rfloor$ if\/ $n$ is large enough.
\end{theorem}

For a better overview, we include tables with previous results, our results and open cases for trees up to 7 vertices. $SD_{2,2}$ denotes the tree on 7 vertices obtained from the double star $D_{2,2}$ by subdividing the edge connecting its two centers.

\section{Constructions}

\begin{proof}[Proof of Proposition \ref{parameter}]
For all lower bounds we need constructions.

To see (\ref{longestpath}), we use the construction of Kopylov and Balister, Gy\H ori, Lehel, and Schelp: let $G_{n,k,s}$ be the graph defined by partitioning $V=X\cup Y\cup Z$ with $|X|=k-2s$, $|Y|=s$, and $|Z|=n-k+s$ such that $G[X \cup Y]$ is a clique and the set of all other edges of $G_{n,k,s}$ is $\{(yz):y\in Y, z\in Z\}$. If $k>2s$, then $G_{n,k,s}$ is $P_{k+1}$-free. Plugging $k=\ell(T)$ and $s=\lfloor \frac{\ell(T)-1}{2}\rfloor$ proves the claim.

For the general lower bound of (\ref{inducedpath}), we construct a graph $G(V,E)$ as follows: let $s:=\lfloor \frac{n}{k-p(T)-1}\rfloor$ and let $V$ be partitioned into $K_1\cup P_1\cup K_2\cup P_2 \cup \dots \cup K_s \cup P_s$ with $|K_i|=k-2p(T)-3$ for all $1\le i \le s$, $|P_i|=p(T)+1$ for all $1\le i<s$. $G[K_i]$ is a clique for all $i$. Every clique $K_i$ contains a special vertex $x_i$, and $G[\{x_i,x_{i+1}\}\cup P_i]$ is a path with end vertices $x_i$ and $x_{i+1}$ (with $x_{s+1}=x_1$). Then $G$ cannot contain $T$, as a partial copy of $T$ could contain the vertices of a $K_i$ and then at most $p(T)$ vertices from both of $P_{i-1}$ and 
$P_i$, so at least one vertex of $T$ cannot be embedded.

To see the furthermore part of (\ref{inducedpath}), we have the following construction $G$: we partition the vertex set of $G$ into $\{v\} \cup \bigcup_{i=1}^s(C_i\cup P_i)$, where $s=\lceil \frac{n-1}{k}\rceil$ with $|C_i|=k-p(T)-1$, $|P_i|=p(T)+1$ for all $1\le i < s$, and $|P_i|\le p(T)+1$ and if $|C_i|>0$, then $|P_i|=p(T)+1$. The edges of $G$ are defined such that $G[\{v\}\cup \bigcup_{i=1}^sP_i]$ is a spider with center $v$ and legs $P_i$, $G[C_i]$ is a clique and exactly one vertex of $C_i$ is connected to the leaf of the leg in $P_i$. The number of edges adjacent to $C_i \cup P_i$ is $\binom{k-p(T)-1)}{2}+p(T)+2$, therefore $e(G)$ is as claimed. Finally, to see that $G$ is $T$-free, observe that as $T$ contains at least two vertices of degree at least 3, if $G$ contained a copy of $T$, then this copy should contain a vertex $u$ from one of the $C_i$s. Also, such a copy cannot contain all vertices of $P_i$ as $p(T)<|P_i|$. Therefore, the vertices of the copy of $T$ should be contained in $|C_i|+|P_i|-1<k$ vertices - a contradiction.

To see (\ref{maxdeg}), it is known that  there exist connected $k$-regular graphs on $n$ vertices if $nk$ is even and there exist connected $n$-vertex graphs with all but one vertex having degree $k$ and the remaining vertex degree $k-1$ if $nk$ is odd. A connected $(\Delta(T)-1)$-regular  or nearly $(\Delta(T)-1)$-regular graph clearly does not contain $T$.

The lower bound of (\ref{matching}) is shown by $K_{\nu(T)-1}+E_{n-\nu(T)+1}$ that has matching number $\nu(T)-1$ and therefore cannot contain $T$.

The lower bound of (\ref{secondsmallest}) is shown by the following construction of a connected $n$-vertex $T$-free graph $G$: we partition the vertex set of $G$ into $\{v\}\cup \bigcup_{i=1}^{\lceil \frac{n-1}{k-1}\rceil}(A_i \cup \{x_i\})$ with $|A_i|=k-2$ for all $i=1,2,\dots, \lfloor \frac{n-1}{k-1}\rfloor$. The edges of $G$ are defined as follows: $G[A_i]$ is a clique, $v$ is adjacent to all $x_i$, and $x_i$ is adjacent to $\delta_2(T)-2$ vertices of $A_i$, so $d_G(x_i)=\delta_2(T)-1$. We claim that $G$ is $T$-free. Indeed, as $G-v$ has components of size at most $k-1$, a copy of $T$ must contain $v$. As $T$ is not a star, at least one of $v$'s neighbors is not a leaf and so its degree should be at least $\delta_2(T)$. But all $v$'s neighbors are $x_i$ vertices that have degree $\delta_2(T)-1$ in $G$.

The construction yielding the lower bound of (\ref{bipartition}) is $K_{a-1,n-a+1}$ as it does not contain bipartite graphs with both parts having at least $a$ vertices.

The construction yielding the lower bound of (\ref{largestbranch}) is $G=K_1+(rK_{m(T)-1}\cup K_s)$, where $r=\lfloor \frac{n-1}{m(T)-1}\rfloor$ and $s\ge 0$. Indeed, if $G$ contained a copy of $T$, then this copy should contain the vertex $v$ of $K_1$ as otherwise $T$ would be contained in $m(T)-1$ vertices. But then we cannot embed the largest branch pending on $v$ as it has size at least $m(T)$. 

To obtain the construction yielding the lower bound of (\ref{twobranches}), we partition the vertex set to $A_1,A_2,\dots,A_s,A_{s+1}$ with $s=\lfloor \frac{n}{k-m_2(T)}\rfloor$ and $|A_i|=k-m_2(T)$ for all $i=1,2,\dots,s$. As $T$ is not a path, we have $k-m_2(T)\ge 2$, so in each $A_i$ we can pick two distinct vertices $x_i,y_i$, maybe with the exception of $A_{s+1}$. Then we define $G$ as a ``cycle of cliques", so $G[A_i]$ is a clique for all $i$, and $x_iy_{i+1}$ is an edge (formally there should be three cases depending whether $A_{s+1}$ has size 0, 1, or at least 2). To see that $G$ is $T$-free, consider the vertex $v$ with $m_2(T)=m_{T,2}(v)$, i.e.\ the largest two components $C_1,C_2$ in $T-v$ have a total size of $m_2(T)$. Suppose $G$ contains a copy of $T$ and the vertex playing the role of $v$ belongs to $A_i$. Then, as there are only two edges leaving $A_i$, $T$ apart from two components of $T-v$ must be embedded into $A_i$. Moreover, since the two edges leave from distinct vertices, at least one vertex of the two exceptional components must also be embedded to $A_i$. So $A_i$ should contain at least $k-m_2(T)+1$ vertices --- a contradiction. (If $i=s+1$ and $x_i=y_i$, then we have the same contradiction, as then $A_{s+1}$ should contain at least $k-m_2(T)$ vertices, but $A_{s+1}$ is strictly smaller than that.)

The construction yielding the lower bound of (\ref{weight}) is $K_{w(T)-1,n-w(T)+1}$ as all its edges have weight $w(T)-1$ and thus $K_{w(T)-1,n-w(T)+1}$ cannot contain $T$. 
\end{proof}

\section{Proofs}

We start by proving Theorem \ref{gamma}. It will be a consequence of the following two results.

\begin{theorem}\label{generalbound}
For any tree\/ $T$ on\/ $k$ vertices, we have\/ $ex_c(n,T)\ge \lfloor\frac{k}{6}\rfloor n$ if\/ $n$ is large enough.
\end{theorem}

\begin{proof} \
\textsc{Case I: } $m(T)>\lfloor k/3\rfloor$.

\vskip 0.2truecm

Then by Proposition \ref{parameter} (\ref{largestbranch}) we have $$\ex_c(n,T)\ge n-1+ \left\lfloor \frac{n-1}{m(T)-1}\right\rfloor \binom{m(T)-1}{2}\ge (n-1)\left(1 +\frac{\lfloor k/3\rfloor-1}{2}\right)\ge n\left\lfloor\frac{k}{6}\right\rfloor,$$
if $n$ is large enough.

\vskip 0.3truecm 

\textsc{Case II: } $m(T)\le \lfloor k/3\rfloor$.

\vskip 0.2truecm

Let $v$ be a vertex such that $T-v$ contains only components of size at most $\lfloor k/3\rfloor$ for some vertex $v$. Let $n=s\lfloor k/3\rfloor +r$ with $r<\lfloor k/3\rfloor$. Consider the graph $G$ on vertex set $A_1\cup A_2\cup\dots\cup A_{s+1}$ with $|A_i|=\lfloor k/3\rfloor$ for all $1\le i\le s$ and $|A_{s+1}|=r$ such that $G[A_i]$ is a clique for all $i$ and $G$ contains one edge between $A_i$ and $A_{i+1}$ for all $i$. Again $e(G)\ge  \lfloor\frac{ k}{6}\rfloor n$ and  $T$ cannot occur as a subgraph in $G$.

Indeed, if there was a copy of $T$ in $G$, then the block containing $v$ splits the path structure of the $A_i$s into two parts, and at least one of them should contain at least $(|T|-\lfloor k/3\rfloor)/2$ vertices of $T$, which is more than $\lfloor k/3\rfloor$, hence contradicting the choice of $v$.
\end{proof}

The broom, which we denote by $B(k,a)$, is the special spider $S_{a-1,1,1,\dots,1}$ on $k$ vertices.

\begin{theorem}\label{broom}\

\begin{enumerate}
    \item 
    For any\/ $a\le k-2$, we have\/ $\ex_c(n,B(k,a))\ge \max\{\lfloor \frac{(k-a)n}{2}\rfloor,\lfloor \frac{a-1}{2}\rfloor(n-\lfloor \frac{a-1}{2}\rfloor)\}$.
    \item
     For any\/ $a\le k/3$, we have\/ $\ex_c(n,B(k,a))= \lfloor \frac{(k-a)n}{2}\rfloor$ if\/ $n$ is large enough.
\end{enumerate}
\end{theorem}

\begin{proof}
The lower bound of $\lfloor \frac{(k-a)n}{2}\rfloor$ follows from Proposition \ref{parameter} (3), while, as $\ell(B(k,a))=a+1$, Proposition \ref{parameter} (1) yields the lower bound $\lfloor \frac{a-1}{2}\rfloor(n-\lfloor \frac{a-1}{2}\rfloor)$.

To see the upper bound of (2), let $G(V,E)$ be an $n$-vertex $B(k,a)$-free graph with $a\le k/3$. Assume first that there exists a vertex $x$ with $d_G(x)\ge k-1$. We claim that $G[V\setminus \{x\}]$ does not contain a path on $2a-3$ vertices.
Indeed, suppose to the contrary that $y_1,y_2,\dots,y_{2a-3}$ is a path in $G[V\setminus \{x\}]$. Then as $G$ is connected, there exists a path $P$ from $x$ to some $y_j$ that does not contain any other $y_i$. Then either $x,P,y_j,y_{j-1},\dots,y_1$ or $x,P,y_j,y_{j+1},\dots,y_{2a-3}$ contains at least $a$ vertices. So $x$ and the first $a-1$ of them together with the other neighbors of $x$ form a copy of $B(k,a)$ --- a contradiction. Theorem \ref{paths} implies that if $n$ is large enough, then $e(G)\le n-1 + e(G-x)\le n-1+\lfloor \frac{2a-5}{2}\rfloor n\le an \le \lfloor \frac{k-a}{2}n\rfloor$ as $a\le k/3$. This finishes the proof in this case.

Assume finally that $\Delta(G)\le k-2$. Then if $n$ is large enough, every vertex $x$ of $G$ is the endpoint of a path on $a\cdot k$ vertices, since $G$ is connected and have maximum degree at most $k-2$. Suppose towards a contradiction that $G$ contains a vertex $x$ with $d_G(x)=d\ge k-a+1$. Let $z_1,z_2,\dots z_d$ be the neighbors of $x$ and let $x,y_2,y_3,\dots,y_{a\cdot k}$ be a path $P$. Then $y_2$ is one of the $z_j$'s, and as $d\le k-2$, there must exist $z_j$ such that $z_j\in P$, say $z_j=y_i$ and either $y_{i-1},y_{i-2},\dots,y_{i-a+2}$ or $y_{i+1},y_{i+2},\dots,y_{i+a-2}$ are not neighbors of $x$. Then $x$, these $y_i$s and the neighbors of $x$ form a $B(k,a)$. 

We obtained that $\Delta(G)\le k-a$ must hold, which implies $e(G)\le \lfloor \frac{(k-a)n}{2}\rfloor$ as claimed.
\end{proof}

\vskip 0.3truecm

\begin{proof}[Proof of Theorem \ref{gamma}]
 The lower bound follows from Theorem \ref{generalbound}, the upper bound from Theorem \ref{broom} (2) with taking $a=\lfloor k/3\rfloor$.
\end{proof}

\vskip 0.3truecm

We continue by proving Theorem \ref{smalltrees}. We restate and prove its parts separately.

\begin{theorem}
For\/ $T=S_{2,1,\dots,1}$, we have\/ $\ex_c(n,T)=\lfloor \frac{n(\Delta(T)-1)}{2}\rfloor$.
\end{theorem}

\begin{proof}
The constructions giving the lower bounds are connected (nearly) regular graphs of degree $\Delta(T)-1$.

If $T=S_{2,1,1,\dots,1}$, then the upper bound proof is a special case of Theorem \ref{broom}, but for completeness, we give a simpler proof of this case.
If $G$ is a connected, $n$-vertex,  $T$-free graph and for some $x$ we have $d_G(x)\ge \Delta(T)$, then $G$ is the star. Indeed, the neighbors of $x$ can be adjacent only to other neighbors of $x$, otherwise $T$ would be a subset of $G$. So by connectivity $N_G[x]=V(G)$. But then if there is at least one edge between two neighbors of $x$, then, as $|V(G)|\ge |V(T)|$, 
again $T$ would be a subgraph of $G$. The star has fewer edges than the claimed maximum, so to have $\ex_c(n,T)$ edges, $G$ must be (nearly) $(\Delta(T)-1)$-regular.
\end{proof}

\begin{theorem}
We have\/ $\ex_c(n,D_{2,2})=2n-4$ for any\/ $n\ge 6$.
\end{theorem}

\begin{proof}
To see the lower bound, observe that $K_{2,n-2}$ is $D_{2,2}$-free as $w(K_{2,n-2})=2$, while $w(D_{2,2})=3$.

 To see the upper bound, observe first that all connected graphs with 6 vertices and at least 9 edges contain  a copy of $D_{2,2}$ as can be checked in the table of graphs of \cite{H} on pages 222--224.
 
Suppose there exists a minimum counterexample: a connected graph $G$ on $n\ge 7$ vertices and $e(G)\ge 2n-3$ edges with no copy of $D_{2,2}$.  We consider several cases.
 
\vskip 0.3truecm 
 
\textsc{Case I: }  $\delta(G)\le 2$ and there is a vertex $v$ of degree at most 2 which is not a cut-point.

\vskip 0.2truecm

Delete a vertex $v$ of degree 1 or 2   to obtain a connected $H =G\setminus v$ with $|H|\ge 6$. By minimality  $e(H) \le  2(n-1) – 4$  and $2n-3\le e(G) \le e(H) +2 \le 2(n-1) – 4  +2 = 2n-4$, a contradiction.

\vskip 0.3truecm

\textsc{Case II: } $\delta(G)=2$ and every vertex of degree 2 is a cut-point.

\vskip 0.2truecm

Consider $v$ of degree 2 such that in $H =  G - v$ out of the two components $A$ and $B$, $|A|$ is as small as possible. Let $w$ be the vertex in $A$ adjacent to $v$ and let $z$ be the vertex in $B$ adjacent to $v$.  

If $|A|\ge 6$ then  by minimality of $G$,  $2n – 3 \le e(G) \le 2|A| -  4 +2|B| - 4  +2  =  2(|A|+|B| +1)  - 8  = 2n-8$, a contradiction. Otherwise  $3 \le |A| \le 5$ as $|A|\le 2$ would imply $\delta(G)=1$ and we were in Case I.
Also, $|A|\ge 4$ as $|A| = 3$ would imply that $A$ must contain a vertex of degree 2 which is not a cut-point  and we were in Case I again.

Suppose  $|A| = 5$. If $d_G(w)= 2$ then $|A|$ is not minimum, so in the induced subgraph on $A$ all vertices have degree at least 2 and $d_G(w)  \ge 3$.
But then the induced graph on $A$ either contains a vertex of degree 2  which is not a cut-point and we are in Case I or all degrees in $G[A \cup \{v\}]$ (except for $v$) are at least 3. Then one can find a copy of $D_{2,2}$ with $w$ being one of the centers and $v$ being a leaf pending from $w$. Indeed, by the degree condition, $G[A\setminus \{w\}]$ contains a $C_4$, so if $N(w)$ contains two non-neighbor vertices $x,y$ of this $C_4$, then $x$ can be the other center of the copy of $D_{2,2}$ and $y$ the other leaf pending from $w$. Otherwise $w$ has exactly two neighbors in $A$, and then by the degree condition $G[A\setminus \{w\}]$ is $K_4$ and it is trivial to embed $D_{2,2}$.

Finally suppose  $|A| = 4$.  As $|B| \ge |A|= 4$,   it follows that  $B^* =  B \cup \{ v , w\}$ has at least 6 vertices and $|B^*| = n-3$,  and hence by  minimality of $G$,  $e(B^*)$ contains at most $2(n – 3)  - 4$ edges  and together with at most 6 edges in $A$ gives $e(G)\le  2n -10 +6 = 2n – 4$ --- a contradiction.
 

\vskip 0.3truecm

\textsc{Case III: } $\delta(G) \ge 3$. 

\vskip 0.2truecm

If all vertices are of degree 3, we have $3n/2$ edges, which is at most $2n – 4$ for $n \ge 8$.  For $n = 7$ this is impossible by parity, hence  $\delta(G) \ge 3$ and  $\Delta(G) \ge 4$.
Consider an edge $e =xy$ with $d_G(y)  = \Delta(G)  \ge 4$  and $d_G(x)\ge 3$.

If $d_G(y) \ge 5$, then for $u,u'\in N(x)$ we have $|N(y)\setminus \{x,u,u'\}|\ge 2$, so $x$ and $y$ are centers of a copy of $D_{2,2}$.
If $d_G(y) = 4$ and $d_G(x) = 4$ then either $x$ and $y$ have distinct neighbors $s$  not in $N[y]$ and $t$ not in $N[x]$ and we find a copy of $D_{2,2}$ with centers $x,y$, or $x$ and $y$  are twins having the same neighbors  $a,b,c$ excluding themselves.
But as $|G| \ge 7$, at least one vertex, say $a$, has a neighbor $d$ not adjacent  to the other 4 vertices and then $a$ and $x$ can be centers of $D_{2, 2}$ with $y$ and $d$ pending from $a$.

So we can assume that all vertices have degree 3 or 4 and vertices of degree 4 form an independent set $Q$.  Let $P  = V  \setminus Q$,  and consider the bipartite  $G[P,Q]$ where $p +q  = n$,  $|P| = p$  and $|Q| = q$.   Clearly,  $4q= e(P,Q) \le 3p$.
 Hence $3n  = 3q +3p  \ge 7q$  and $q \le 3n/7$,  $p \ge 4n/7$.  But then 
 $$e(G) = \frac{4q +3p}{2}  \le \frac{12n/7 +  12n/7}{2}  = \frac{12n}{7}  < 2n- 3$$ for $n\ge 11$. So we are left with $n = 7, 8 ,9 ,10$.
 
 For $n = 7$:  $q \le 3n/7=3$ and $q$ must be an integer. If $q=3$, then $G = K_{4,3}$ containing $D_{2,2}$. The case $q = 2$ is impossible as the degree sum would be odd (by the number $p$ of odd-degree vertices). Hence $q = 1$ and $p = 6$. Consider a vertex $v$ of degree 4  and its neighbors  $a,b,c,d$ all of degree 3.  
If say  $a$ is  adjacent to a vertex outside $\{v, b,c,d \}$, then there is $D_{2,2}$. But as this holds for all of $a,b,c,d$  it means $A = \{  v,a,b,c,d \}$  has no neighbor in $V \setminus A$ and $G$ is not connected.     

For $n = 8$, we still have $q \le \lfloor \frac{3n}{7}\rfloor=3$ and $p \ge 5$. But $p = 5,7$ are impossible, again due to parity, hence $q = 2$ and $p  = 6$.  
Let $Q = \{ a,b \}$ be  the set of vertices of degree 4, and let $P = V\setminus Q$.   If some vertex $x$ in $P$ is adjacent to both $a$ and $b$, then  consider the only neighbor $z$ of $x$ in $P$. Here $a$ is adjacent to $x$ and three more vertices in $P$, so at least two vertices except $x$ and $z$ are neighbors of $a$ and $x$ can use $z$ and $b$ to obtain a copy of $D_{2,2}$ with centers $x$ and $a$. Hence every vertex in $P$ is adjacent to at most one vertex in $Q$, yielding $|P|\ge e(P,Q) =2|Q|$ --- a contradiction.   

For  $n = 9$, we have  $q \le \lfloor \frac{3n}{7}\rfloor=3$. The case $q = 2$ is impossible by parity and $q = 1$, $p = 8$ implies $e(G) = (4 +24)/2 = 14  =  2n  - 4$ as stated by the theorem. So only $q = 3$, $p = 6$ is to be checked.  
Let $Q=  \{ a,b,c \}$ be  the set of vertices of degree 4, and let $P = V \setminus Q$. If some vertex $v$ in $P$ has at least two neighbors in  $Q$, say $a,b$, then we have a copy of $D_{2,2}$ with centers $v$ and $a$, as all the four neighbors of $a$ are in $P$ and at most two of them belong to $N[x]$. So  every vertex in $P$ can have at most one neighbor in $Q$ and as in the previous case we have $|P|\ge e(P,Q)=4|Q|$ --- a contradiction.   

For $n = 10$,  $q \le \lfloor \frac{3n}{7}\rfloor=4$, and so parity of the degree sum implies $q = 4$ or $q = 2$. If $q = 2$ then $e(G)  = (8 + 24)/2  = 16= 2n-4$ as stated in the theorem,  so only $q = 4$, $p = 6$ remains  to be checked.

Let $Q = \{ a,b,c,d \}$  be the set of vertices of degree 4, and let $P = V \setminus Q$. If some vertex $v$ in $P$ has all its neighbors in $Q$, say $a,b,c$, then  we  obtain a copy of $D_{2,2}$ with centers $v$ and $a$. 
Otherwise, we have $4|Q|=e(P,Q)\le 2|P|$, a contradiction.  
\end{proof}

\begin{theorem}
$\ex_c(n,S_{3,1,1})=\lfloor \frac{3(n-1)}{2}\rfloor$ if\/ $n\ge 7$ and\/ $ex(6,S_{3,1,1})=9$.
\end{theorem}

\begin{proof}
The lower bounds are shown by $K_1+M_{n-1}$ for $n\ge 7$ and by $K_{3,3}$ for $n=6$. The former is $S_{3,1,1}$-free as shown in Proposition \ref{parameter} (\ref{largestbranch}) with $m(S_{3,1,1})=3$. The graph $K_{3,3}$ is $S_{3,1,1}$-free as the bipartition of $S_{3,1,1}$ has a part of size 4.

To obtain the upper bound, we consider an $S_{3,1,1}$-free connected graph $G$. The general idea is to choose a longest cycle $C=v_1v_2,\dots,v_k$ in $G$, and argue depending on its length $k$.
 
If $k=n$, then $C$ is a Hamiltonian cycle. It cannot have short chords; e.g.\ if $v_2v_4$ is an edge, then $S_{3,1,1}$ can have center $v_2$ and legs $v_2v_1$, $v_2v_3$, $v_2v_4v_5v_6$. Moreover if $n>6$, then longer chords cannot occur either. Indeed, if $v_2v_j$ with $j=5,…,n-2$ is an edge, then $v_2$ with $v_j$ and its two successors can form the leg of length 3. Likewise for $j=6,…,n-1$ such a leg can be formed using the two predecessors of $v_j$, still keeping the legs $v_2v_1$ and $v_2v_3$. This excludes all chords if $n>6$, hence $|E(G)|=n$. If $n=6$, then antipodal vertices can be adjacent without creating any copy of $S_{3,1,1}$, but no other chords may occur. In this way we obtain the extremal graph $K_{3,3}$.
 
Assume next that $4<k<n$. We show that this is impossible whenever $n\ge 6$. Since $G$ is connected, there is a vertex $x$ not in $C$ but having at least one neighbor in $C$. If e.g.\ $xv_2$ is an edge, we find $S_{3,1,1}$ with center $v_2$ and legs $xv_2$, $v_2v_1$, $v_2v_3v_4v_5$.
 
Assume now $k=4$, $C=v_1v_2v_3v_4$, $n\ge 6$. If $P$ is any path with one end in $C$ and all its other vertices in $V(G)\setminus V(C)$, then $P$ can have no more than two edges, otherwise $S_{3,1,1}$ would be found, with the long leg in $P$ and the two short legs in $C$. We are going to prove that if $P$ is shorter than 3, the number of edges in $G$ is smaller than what is given in the theorem.
 
If $P$ has length 2, let $xyv_1$ be a path attached to $C$. Then the edges $xv_2$, $xv_3$, $xv_4$, $yv_2$, $yv_4$ cannot be present because $C$ is a longest cycle. Also the edges $v_1v_3$ and $v_2v_4$ are excluded because $G$ is $S_{3,1,1}$-free. This implies $|E(G)|\le 8$ if $n=6$. If $n>6$, there should be a further vertex $z$ adjacent to $C\cup P$, but any edge from $z$ to $C\cup P$ would create an $S_{3,1,1}$. (For $zx$ the center is $v_1$, and for any other edge the center is the neighbor of $z$.) Hence $n>6$ is impossible in this case.
 
Suppose that $P=yv_1$ is a single edge not extendable to a longer path outside $C$. Then a sixth vertex $x$ can only be adjacent to $v_2$ or $v_4$ (or both), otherwise an $S_{3,1,1}$ would occur. And also here, it is not possible to extend this graph to a connected graph of order 7 without creating an $S_{3,1,1}$ subgraph. Hence $n=6$. Moreover, the diagonals of $C$ must be missing; e.g.\ the edges $xv_2$ and $v_2v_4$ would yield $S_{3,1,1}$ with center $v_2$ and legs $xv_2$, $v_2v_3$, $v_2v_4v_1y$. Thus the number of edges is only 4 plus the degree sum of $x$ and $y$, which is at most 7 because the presence of all four edges $xv_2$, $xv_4$, $yv_1$, $yv_3$ would make $G$ Hamiltonian, hence $C$ would not be a longest cycle.
 
Finally we have to consider graphs without any cycles longer than 3. It means that each block of $G$ is $K_2$ or $K_3$. Let $f(n)$ denote the maximum number of edges in such a graph. We clearly have $f(1)=0$, $f(2)=1$, $f(3)=3$. Let $B$ be an endblock of $G$, with cut vertex $w$. Deleting $B-w$ from $G$ we obtain a $S_{3,1,1}$-free connected graph of order $n-|V(B)|+1$, where $|V(B)|$ is 2 or 3. Hence
$$	f(n) \le \max \{ f(n-1) + 1 , f(n-2) + 3 \}.$$
This recursion implies $f(n) \le \lfloor 3(n-1)/2 \rfloor$ for every $n$, completing the proof of the upper bound for $n\ge 7$.
\end{proof}

\begin{theorem}
$\ex_c(n,S_{2,2,1})=2n-3$ if\/ $n\ge 6.$
\end{theorem}

\begin{proof}
The lower bound is shown by $K_2+E_{n-2}$ as it has matching number 2, while $\nu(S_{2,2,1})=3$.

To obtain the upper bound on $\ex_c(n,S_{2,2,1})$, we proceed by induction: for $n = 6$ every connected graph on 6 vertices and 10 edges contains $S_{2,2,1}$ (by inspecting the table of graphs of \cite{H} on pages 222--224).
 
For the induction step assume that the statement of the theorem holds for graphs of at most $n-1$ vertices and assume on the contrary that $G$ is a connected graph on $n$ vertices and $2n-2$ edges without $S_{2,2,1}$. Here $2n-2$ suffices as otherwise if $e(G) \ge 2n – 1$, we can delete an edge on a cycle.  
 
If $\delta(G) \le 2$ and there is a vertex $v$ of degree at most 2 which is not a cut-point, then we can apply induction to $H=F-v$ to obtain $e(G)\le e(H)+2\le 2(n-1)-3+2=2n-3$, a contradiction.
 
Suppose $\delta(G) = 2$ and every vertex of degree 2 is a cut-point. Then let $v$ be such a cut-point with neighbors $x$ and $y$. Consider $H =  G - v + (xy)$. Here $|H| = n-1$ and $e(H) = 2n-2 -2 +1 = 2(n-1) -2 +1$, hence by induction $H$ contains a copy $S$ of $S_{2,2,1}$. If $S$ does not use the edge $xy$, then $S$ is also in $G$ --- a contradiction. If $S$ uses $xy$ such that one of $x$ and $y$, say $x$, is a leaf in $S$, then replace $x$ by $v$ and the edge $xy$ by $vy$ to obtain a copy $S'$ of $S_{2,2,1}$ in $G$ --- a contradiction. Finally, if $xy$ is the edge of a 2-leg of $S$ containing the center, say $x$ and the leg is $xyz$, then replace this leg by $xvy$ to obtain $S'$ in $G$ --- a contradiction.
 
So we can assume $\delta(G)  \ge 3$. If all vertices are of degree 3, then $e(G)  = 3n/2 < 2n  - 2$. If all vertices are of degree at least 4, then $e(G)  \ge 2n>2n-2$, hence there exists a vertex $y$  of degree 3 adjacent to a vertex $x$ of degree at least 4. Let $u,v$ be the other two neighbors of $y$, and let $z\neq u,v,y$ be a neighbor of $x$. If $u$ or $v$ has a neighbor outside these 5 vertices, then we obtain a copy of $S_{2,2,1}$ with center $y$. If not and $N(x)=\{u,v,y,z\}$, then $z$ must have a neighbor outside these 5 vertices and we obtain a copy of $S_{2,2,1}$ with center $x$. Finally, if $N(u)\cup N(v)\subseteq \{u,v,x,y,z\}$ and $z'$ is another neighbor of $x$, then $d_G(z')\ge 3$ implies that $z'$ must have a neighbor outside these 6 vertices, and we obtain a copy of $S_{2,2,1}$ with center $x$. This contradiction finishes the proof.
\end{proof}

\begin{proof}[Proof of Theorem \ref{d22*}]
 The assertion is trivial for $n<7$. For larger $n$ the split graph construction $K_2 + E_{n-2}$ shows that $2n-3$ is a lower bound.
 
To derive the same as an upper bound, assume $n>6$ and consider any $D^*_{2,2}$-free graph $G$ of order $n$ with more than $2n-4$ edges. Then, by Theorem \ref{smalltrees} (\ref{d22}), there is a $D=D_{2,2}$ subgraph in $G$; let the central edge of $D$ be $xy$.

If some vertex not in $D$ is adjacent to a leaf of $D$, then a copy of $D^*_{2,2}$ arises, a contradiction. More generally, there cannot exist any vertex at distance exactly 2 from $\{x,y\}$. By the connectivity of $G$, it follows that every vertex of $G$ is adjacent to at least one of $x$ and $y$. On this basis we partition $V(G)-\{x,y\}$, defining
\[
X = N(x)-N[y],	\hskip 1truecm Y = N(y)-N[x], \hskip 1truecm	Z = N(x)\cap N(y) .\]
Let us assume $|Y|\ge |X|$. Due to the presence of $D_{2,2}$ we know that $|X|+|Z|\ge 2$ holds. Moreover, $|Y|\ge |X|$ with $n\ge 7$ implies $|Y|+|Z|\ge 3$. Hence there cannot be any $X-Y$ edges, moreover $Y\cup Z$ is an independent set, both because $G$ is $D^*_{2,2}$-free. For the same reason, if $|X|+|Z|>2$, then also $X\cup Z$ is independent. In this case the entire $X\cup Y \cup Z$ is independent and $G$ cannot have more than $2n-3$ edges, yielding just the extremal split graph $K_2 + E_{n-2}$. Otherwise, if $|X|+|Z|=2$, there can be just one edge inside $X\cup Z$, hence we have 6 edges in the $K_4$ subgraph induced by $X\cup Z \cup \{x,y\}$, and there are further $n-4$ edges from $Y$ to $y$. These are altogether $n+2$ edges only, i.e.\ fewer than the assumed $2n-3$. This contradiction completes the proof.
\end{proof}

\begin{proof}[Proof of Theorem \ref{s222}]
 To simplify notation, let $f(n) = \ex_c(n,S_{2,2,2})$. The lower bound for $n\le 7$ is obtained by the following construction that works for all $n$. Take a complete graph $K_4$ on the vertex set $\{v_1,v_2,v_3,v_4\}$ and join all $v_i$ for $i=5,6,…,n$ to $v_1$ and $v_2$. Equivalently, $v_1$ and $v_2$ are universal vertices, supplemented with the single edge $v_3v_4$. This connected graph with $2n-2$ edges does not contain $S_{2,2,2}$ because it is not possible to delete two vertices from $S_{2,2,2}$ to destroy all but one edges.
 
The argument for the upper bound applies induction on $n$, with basic cases $n\le 7$, from which only $n=7$ is nontrivial. We note here that $n=5$ and $n=6$ are the only cases where $2n-2$ is not an upper bound on the formula given for $f(n)$.
 
For $n=7$ the assertion is that every connected graph $G$ with 7 vertices and at least 13 edges contains $S_{2,2,2}$ as a subgraph. To prove it, suppose first that G has a cut-point $x$, and consider the vertex distribution between the components of $G-x$. If it is (3,3) --- where we unite components if there are more than two, e.g.\ the distribution (3,2,1) is also viewed as (3,3) --– then already 9 nonadjacencies are found, hence $G$ would have at most $21-9=12$ edges, a contradiction. If the distribution is (2,4), then it forces 8 nonadjacencies, hence $G$ must be the graph in which the two blocks incident with $x$ are $K_3$ and $K_5$. Obviously this graph contains $S_{2,2,2}$. If the distribution is (1,5), then $x$ has a pendant neighbor, say $y$, and $G-y$ is a connected graph of order 6, having at least 12 edges. Routine inspection shows that all such graphs $G$ contain $S_{2,2,2}$.
 
Assume that $G$ is 2-connected. If $G$ has minimum degree 3, then $G$ has a Hamiltonian cycle, say $C=v_1v_2v_3v_4v_5v_6v_7$. (More generally it is well known that a graph of order $2d+1$ and minimum degree $d$ is non-Hamiltonian if and only if either it is the complete bipartite graph $K_{d,d+1}$ or it has two blocks incident with a cut vertex, both blocks being $K_{d+1}$; in our case both of them would have only 12 edges.) The presence of any long chord in $C$, e.g.\ $v_3v_6$ immediately creates an $S_{2,2,2}$ with center $v_3$ and legs $v_3v_2v_1$, $v_3v_4v_5$, $v_3v_6v_7$. Moreover, any three consecutive short chords, e.g.\ $v_2v_4$, $v_3v_5$, $v_4v_6$ create an $S_{2,2,2}$ with center $v_4$ and legs $v_4v_2v_1$, $v_4v_3v_5$, $v_4v_6v_7$. And now at least one of these situations holds because in general a cycle of length $n$ without three consecutive short chords and with no other chords at all can have no more than $n+2n/3<2n-2$ edges if $n\ge 7$.
 
Hence in the 2-connected case $G$ has minimum degree exactly 2, and if we remove a vertex $x$ of degree 2, we obtain a graph on 6 vertices with at least 11 edges. If it is $K_5$ with a pendant edge, then the pendant vertex must be adjacent to $x$ and we immediately find $S_{2,2,2}$. Otherwise there can be at most one vertex of degree 2 in $G-x$, hence it contains a $C_6$, say $v_1v_2v_3v_4v_5v_6$ (as a rather particular corollary of P\'osa’s theorem). If the two neighbors of $x$ are antipodal in $C$, e.g.\ $v_3$ and $v_6$, we find $S_{2,2,2}$ with center $v_3$ and legs $v_3xv_6$, $v_3v_2v_1$, $v_3v_4v_5$. If the two neighbors of $x$ are consecutive in $C$, then $C$ extends to $C_7$ which we already settled. Hence we can assume that the neighbors of $x$ are $v_2$ and $v_4$. Since $C$ has at least 5 chords, some of the five chords $v_1v_3$, $v_1v_4$, $v_2v_5$, $v_3v_5$, $v_3v_6$ must be present, and each of them creates $S_{2,2,2}$ with $x$ and the edges of $C$. This completes the proof of $f(7)=12$.
 
Turning now to the inductive step, assume that $n\ge 8$ and that the upper bound $2n-2$ is valid for all smaller orders other than 5 and 6. Depending on the structure of the graph under consideration, we will apply one of the following upper bounds:
 \[
	f(n-1) + 2 ,	\hskip 1truecm	f(n-3) + 6 , \hskip 1truecm		f(n-6) + 12 .
\] 
Suppose that $G$ is an $S_{2,2,2}$-free connected graph of order $n\ge 8$, and $G$ is $S_{2,2,2}$-saturated, i.e.\ the insertion of any new edge inside $V(G)$ would create an $S_{2,2,2}$ subgraph. Under the latter assumption we observe the following.
 
\begin{claim}\label{claim}
 If $x$ is a vertex of degree 2, say with neighbors $y$ and $z$, then $yz$ is also an edge of $G$.
\end{claim}	
 
\begin{proof}[Proof of Claim]
Otherwise $yxz$ would be an induced path in $G$. Let then $G’$ be the graph obtained by the insertion of edge $yz$. By assumption there is an $S=S_{2,2,2}$ subgraph in $G’$, which necessarily contains the edge $yz$. If $yz$ is a leaf edge of $S$, then of course the degree-3 center of $S$ cannot be $x$, it must be another vertex $w$ adjacent to $y$ or to $z$. But then $z$ or $y$ is a leaf vertex of $S$, and replacing $yz$ with $yx$ or $zx$ we find another copy of $S_{2,2,2}$ which is a subgraph of $G$, a contradiction. The other possibility would be that $y$ or $z$ is the degee-3 vertex of S, and the edge $yz$ is continued with a leaf edge $zw$ or $yw$ (allowing also $w=x$). But then $x$ cannot be a mid-vertex of any leg of S since $x$ does not have a neighbor other than $y$ and $z$. Hence the leg $yzw$ or $zyw$ can be replaced with $yxz$ or $zxy$, and we would again find a copy of $S_{2,2,2}$ as a subgraph of $G$.
\end{proof}

As a consequence of Claim \ref{claim}, if $G$ has a vertex of degree 1 or 2, then $|E(G)| \le f(n-1) + 2 \le 2n-2$ follows by induction, because deleting a vertex of minimum degree the graph remains connected. Hence from now on we may assume that $G$ has minimum degree at least 3.
 
Let $C=v_1v_2v_3v_4\dots v_s$ be a longest cycle in $G$. We have already seen that if $s=n$, then $|E(G)| \le 5n/3 < 2n-2$. Next, we observe that if $n>s\ge 5$, then $V(G)\setminus V(C)$ is an independent set. Indeed, if $xy$ is an edge outside $C$ then there is a path $P$ (possibly an edge) from $\{x,y\}$ to $C$ and in this case a copy of $S_{2,2,2}$ is easily found using edges of $C$, with two edges from $P\cup \{xy\}$. E.g., if $v_3x$ is an edge, then $S_{2,2,2}$ can have center $v_3$ and legs $v_3xy$, $v_3v_2v_1$, $v_3v_4v_5$. Thus, every vertex outside of $C$ has at least three neighbors in $C$. Moreover, no two of those neighbors are consecutive in $C$, because $C$ is longest. This immediately excludes $s=5$. But also $s>5$ is impossible because if e.g.\ $v_2$, $v_4$, $v_6$ are neighbors of $x$, then an $S_{2,2,2}$ can have center $x$ and legs $xv_2v_1$, $xv_4v_3$, $xv_6v_5$.
 
As a consequence, investigations are reduced to $S_{2,2,2}$-free connected graphs with minimum degree 3 and without any cycles longer than 4. Such a graph $G$ cannot be 2-connected (because due to Dirac’s theorem, 2-connectivity would imply the presence of a cycle longer than 5). Hence $G$ contains at least two endblocks.
 
Let $B$ be an endblock of $G$, attached with cut-point $w$ to the other part of $G$. We argue that $B$ induces $K_4$ in $G$. All vertices of $B$ except $w$ have degree at least 3 inside $B$, therefore $B$ contains a 4-cycle, say $C’=wxyz$. If there is a vertex $u$ in $V(B)\setminus V(C’)$, then 2-connectivity of $B$ and the exclusion of cycles longer than 4 imply that there are exactly two neighbors of $u$ in $C$, either $w$ and $y$, or $x$ and $z$. But then there must exist a third neighbor $v$ of $u$ not in $C$, and $v$ also has two neighbors in $C$; and then a cycle longer than 4 would occur. Thus $B$ is a $K_4$ indeed.
 
Now we are in a position to complete the proof of the theorem by induction on $n$. Consider any maximal $S_{2,2,2}$-free connected graph $G$ of order $n>7$ that has at least $2n-2$ edges. If $G$ has a vertex of degree at most 2, then apply the upper bound $f(n-1)+2$.

 If $G$ has minimum degree at least 3, we know that $G$ is not 2-connected. Then:
	
If $n=8$ or $n=9$, remove all the 6 non-cutting vertices of two $K_4$ endblocks of $G$ and apply the upper bound $f(n-6)+12$. This yields $|E(G)|\le 13$ for $n=8$ and $|E(G)|\le 15$ for $n=9$, both are smaller than $2n-2$.
	
If $n\ge 10$, remove the 3 non-cutting vertices of a $K_4$ endblock of $G$ and apply the upper bound $f(n-3)+6$. This yields $|E(G)|\le 2n-2$.
\end{proof} 
 
\begin{rem} The extremal graphs are not unique if $n\ge 7$. In the graph constructed at the beginning of the proof we can remove three vertices of degree 2 and attach a block $K_4$ to one of the two high-degree vertices. As another alternative for $n\ge 10$, we can remove six vertices of degree 2 and attach two blocks isomorphic to $K_4$, one block to each high-degree vertex. A further extremal graph of order 7 can be obtained from $K_5$ by attaching two pendant edges to a vertex of $K_5$.
\end{rem} 

\vskip 0.3truecm

\begin{proof}[Proof of Theorem \ref{s321}]
 A lower bound for $n\ge 7$ is the split graph $K_2 +E_{n-2}$ with $2n-3$ edges which does not even contain $S_{2,2,1}$ and hence $S_{3,2,1}$ cannot be a subgraph either.
 
 The proof of the upper bound proceeds by induction on $n$. The base case $n=7$ is left to the Reader.
 
  Aassume $G$ is a minimum connected  counterexample with $n \ge  8$ vertices and  has at least $2n - 2$ edges but no copy of $S_{3,2,1}$.
  
If $G$ contains a vertex $v$ of degree at most 2 such that $H=G-v$ is connected, then, by minimality, $e(H)  \le 2(n-1) - 3$ hence $2n-2 \le e(G)  \le e(H) +2 \le 2n- 3$, a contradiction.
 
 Next, assume $v$ is a cut-point with neighbors $x$ and $y$.  Consider the graph $H$ that we obtain from $G$ by deleting $v$ and adding the edge $xy$.  We will show that if $H$ contains  $S_{3,2,1}$ then so  does $G$.  Let $A$ be the component containing $x$ and $B$ the component containing $y$. By symmetry we may assume that if $H$ contains a copy $S$ of $S_{3,2,1}$, then its center is in $A$ and so $B$ can contain vertices of at most one leg of $S$. We consider cases according to the number of vertices in $S\cap B$. If $A$ contains $S_{3,2,1}$ completely, then so does $G$. If $A$ contains all of $S_{3,2,1}$ except for a leaf played by $y$,  then the same copy with $v$ replacing $y$ is contained in $G$.

If $S\cap B=\{y,w\}$, then the leg of $S$ ending $x-y-w$ can be replaced in $G$ with $x-v-y$  to obtain a copy $S'$ of $S_{3,2,1}$.
If $S\cap B=\{y,w,z\}$, then the leg of $S$ ending $x-y-w-z$ can be replaced in $G$ with $x-v-y-w$  to obtain a copy $S'$ of $S_{3,2,1}$.
 
So, as proved, $H$ must be $S_{3,2,1}$-free, hence $2n-2 \le e(G)  \le e(H) + 1 \le 2(n-1) - 3 +1 \le 2n-4$, a contradiction.
 
Therefore, from now on we may assume $\delta(G) \ge 3$. By Theorem \ref{smalltrees} (\ref{s221}), we know that $G$ contains a copy $S$ of $S_{2,2,1}$. Let $v$ be the center of $S$ with legs $v-u$, $v-x-y$, and $v-a-b$. If $y$ or $b$ has a neighbor not in $S$, then $G$ contains a copy of $S_{3,2,1}$ --- a contradiction.

Suppose $x$ (or $a$)  has a neighbor $z$ not in S.  Then $z$ cannot be adjacent to any of $v,y,a,b$ as a copy of  $S_{3,2,1}$ would appear. Also, $z$ cannot be adjacent to any vertex outside $S$ as again a copy of $S_{3,2,1}$ would appear in $G$.
By $\delta(G)\ge 3$, $z$ must be adjacent to $u,x$, and $a$, but then a copy of $S_{3,2,1}$ (this time with center $z$) would appear in $G$.

We have shown so far that $x,y,a,b$ cannot have neighbors outside $S$. 

If $u$  has at least two neighbors $z$ and $w$ outside $S$, then they cannot be adjacent (it would create the leg $v-u-z-w$ of a copy of  $S_{3,2,1}$) and none of them can have a neighbor outside $S$ as a copy of $S_{3,2,1}$ would appear in $G$.  As shown above, they cannot be adjacent to any of $x ,y,a,b$  hence they have degree at most 2 (with neighbors $u$ and possibly $v$) contradicting $\delta(G) \ge 3$.
 
If $u$ has just one neighbor, say $z$ outside $S$, then $z$ cannot have a neighbor outside $S$ as a copy of $S_{3,2,1}$ would appear, and as before, $z$ cannot be adjacent to any of $x , y , a , b$ hence $z$ can be adjacent to at most $u$ and $v$ but then $d_G(z) \le 2$ contradicts  $\delta(G) \ge 3$. 

So the only vertex of $S$ that can have further neighbors outside $S$ is $v$.  We claim that there cannot exist a path $v-w-z$ with $w,z\notin S$. Indeed, if $w,z$ existed, then any of the edges $ax$, $ay$ would create a copy of $S_{3,2,1}$ with center $a$. Similarly, any of the edges $xa$, $xb$ would create a copy of $S_{3,2,1}$ with center $b$. But then $\delta(G)\ge 3$ implies the
presence of $ua$ and $ux$ in $G$ creating a copy of $S_{3,2,1}$ with center $u$. Therefore all vertices outside $S$ must have degree 1, which case has already been dealt with. This finishes the proof of the induction step.
\end{proof}

\begin{proof}[Proof of Theorem \ref{s31dots1}]
It is enough to prove that if $G$ is a connected $n$-vertex graph with $\Delta(G)\ge \Delta(T)$, then $G$ contains $T$ or $e(G)\le \lfloor \frac{(\Delta(T)-1)n}{2}\rfloor$. So fix a vertex $v$ with $d_G(v)=\Delta(G)\ge\Delta(T)$ and consider the partition $\{v\}, N(v), X:=V(G)\setminus N[v]$.

If $X$ contains an edge $xy$, then by connectivity of $G$, there must exist a path (maybe a single edge) from $xy$ to $N(v)$ and we find a copy of $T$ in $G$. So we may assume that $X$ is independent, and thus by connectivity of $G$, every $x\in X$ is adjacent to at least one $u\in N(v)$.

\vskip 0.3truecm

\textsc{Case I: } $d_G(v)=\Delta(G)>\Delta(T)$.

\vskip 0.2truecm

Then any $x\in X$ is adjacent to exactly one vertex $u\in N(v)$ as if $xu,xu'$ are edges in $G$, then $uxu'$ can form the long leg of a copy of $T$ with center $v$ and other neighbors of $v$ complete this copy of $T$. So $d_G(x)=1$ for all $x\in X$. Let $u,u'\in N(v)$ be two vertices such that at least one of them has a neighbor in $X$. Then again if $uu'$ is an edge, we find a copy of $T$. So if $U\subseteq N(v)$ is the set of neighbors of $v$ that are adjacent to a vertex in $X$ and $U'=N(v)\setminus U$, then $e(G)\le (|\{v\}\cup U\cup X|-1)+ e( U')$. If $|U'|\le\Delta(T)+1$, then $e(U')\le \binom{\Delta(T)+1}{2}$ and so $e(G)\le n-1+\binom{\Delta(T)+1}{2}\le \lfloor \frac{(\Delta(T)-1)n}{2}\rfloor$ as $\Delta(T)-1\ge 3$. Finally, if $|U'|\ge \Delta(T)+2$, then either $G[U']$ is a (partial) matching and thus $e(G)\le (1+|U|+|X|-1)+\frac{3|U'|}{2}\le \frac{3(n-1)}{2}\le \lfloor \frac{(\Delta(T)-1)n}{2}\rfloor$ (here we use $\Delta(T)\ge 4)$ or $G[U']$ contains a path on 3 vertices, and then by $|U'|\ge \Delta(T)+2$ we find a copy of $T$ in $G$.

\vskip 0.3truecm

\textsc{Case II: } $d_G(v)=\Delta(G)=\Delta(T)$.

\vskip 0.2truecm

As $X$ is independent, we have $e(G)\le (\Delta(G)+1)\Delta(G)=(\Delta(T)+1)\Delta(T)=O(1)$.
\end{proof}

\section{Concluding remarks}

Theorem \ref{gamma} gave upper and lower bounds on $\gamma$. If the lower bound of either (1) or (3) of Theorem \ref{broom} turned out  to be (asymptotically) sharp (which we believe to be the case) for $a=(1/2-\varepsilon)k$ or $a=(1/2+\varepsilon)k$, then the upper bound on $\gamma$ would improve from $2/3$ to $1/2$. Note that a special case of Theorem \ref{s31dots1} yields $\ex_c(n,S_{3,1,1,1})=\lfloor \frac{(\Delta(S_{3,1,1,1})-1)n}{2}\rfloor$, so a small case when $a=\lfloor k/2\rfloor$. We have no evidence to believe that the lower bound of $1/3$ on $\gamma$ is best possible. 

In Proposition \ref{parameter}, we enumerated several graph parameters based on which we could define general constructions avoiding trees $T$ for which these parameters have small value. It would be nice to add other parameters to this list, and would be wonderful to prove that it is enough to consider a finite set of parameters to determine the asymptotics of $\ex_c(n,T)$ for all trees $T$. Of particular interest is the characterization of those trees for which $\ex(n,T)  - c(T) \le \ex_c(n,T)  \le \ex(n,T)$ holds for some constant $c(T)$.

As for special tree classes, one such class that could give some insight is the set of spiders with all legs of at most 2 vertices. For the spider $S=S_{2,2,\dots,2,1,1,\dots,1}$ with $t$ legs of two vertices and $s$ legs consisting of a single vertex, we have $|S|=2t+s+1$, and 
\begin{itemize}
    \item 
    $\nu(T)=t+1$ if $s>0$,
    \item
    $\Delta(T)=t+s$,
    \item
    $m_2(T)=4$ if $t\ge 2$.
\end{itemize}
The construction of Proposition \ref{smalltrees} (\ref{maxdeg}) based on maximum degree outperforms the one based on the matching number in Proposition \ref{smalltrees} (\ref{matching}) if $s>t$. But the one based on $m_2$ in Proposition \ref{smalltrees} (\ref{twobranches}) is better than both previous ones once $s\ge 5$ and $t\ge 2$. It would be interesting to see whether these constructions achieve the asymptotics of $\ex_c(n,S)$.

Classical  Turan numbers are monotone with two respects:
Firstly, if $H$ is a subgraph of $F$ then $\ex(n,H) \le \ex(n,F)$.  This inequality is preserved for the connected Tur\'an number $\ex_c(n,F)$ (excluding the small ``undefined" cases $K_2$ and $P_3$).
Secondly, if $m < n$, then $\ex(m,F)\le \ex(n,F)$.  This property is not necessarily preserved by connected Tur\'an numbers  for small values of $n$ with respect to $|T|$. There are several examples given by our results, of the following type: $\ex_c(|T|-1,T)=\binom{|T|-1}{2}>\ex_c(|T|,T)$; see e.g.\ $T=S_{3,2,1}$.

\begin{problem}
Is it true that there exists a threshold $n_0(F)$ such that $\ex_c(m,F)\le \ex_c(n,F)$ holds whenever $n_0(F)\le m<n$?
\end{problem}

\end{document}